\documentclass{amsart}

\newtheorem{Theorem}{Theorem}
\newtheorem*{Theorem*}{Theorem}
\newtheorem*{Proposition*}{Proposition}
\newtheorem{Proposition}{Proposition}
\newtheorem*{Claim}{Claim}

\begin{document}

\title[maximal number and diameter of exceptional surgery sets]{On the maximal number and the diameter of exceptional surgery slope sets}

\address{Department of Mathematics, 
College of Humanities and Sciences, Nihon University, 
3-25-40 Sakurajosui, Setagaya-ku, Tokyo 156-8550, Japan.}

\email{ichihara@math.chs.nihon-u.ac.jp}

\author{Kazuhiro Ichihara}

\thanks{The author is partially supported by
Grant-in-Aid for Young Scientists (B), No.~23740061,
Ministry of Education, Culture, Sports, Science and Technology, Japan.}



\begin{abstract}
Concerning the set of exceptional surgery slopes for a hyperbolic knot, 
Lackenby and Meyerhoff proved that 
the maximal cardinality is 10 and the maximal diameter is 8. 
Their proof is computer-aided in part, 
and both bounds are achieved simultaneously. 
In this note, it is observed that the diameter bound 8 
implies the maximal cardinality bound 10 
for exceptional surgery slope sets. 
This follows from the next known fact: 
For a hyperbolic knot, 
there exists a slope on the peripheral torus such that 
all exceptional surgery slopes 
have distance at most two from the slope. 
We also show that, in generic cases, 
the particular slope above can be taken as 
the slope represented by the shortest geodesic 
on a horotorus in a hyperbolic knot complement. 
\end{abstract}

\maketitle


\section{Introduction}

In the study of 3-manifolds, 
one of the important operations describing 
the relationships between $3$-manifolds 
would be \textit{Dehn surgery}. 
We denote by $K(r)$ 
the resultant 3-manifold via Dehn surgery 
on a knot $K$ along a slope $r$. 
(As usual, by a \textit{slope}, 
we mean an isotopy class of 
a non-trivial unoriented simple loop on a torus.) 
Precisely, the 3-manifold $K(r)$ is obtained by 
removing an open tubular neighborhood $N(K)$ of $K$, 
and gluing a solid torus $V$ back so that 
the slope $r$ on the boundary torus of the complement of $N(K)$ 
is represented by the simple closed curve 
identified with the meridian of $V$.

As a consequence of 
the Geometrization Conjecture, 
raised by Thurston in \cite[Conjecture 1.1]{Th2}, 
and established by celebrated Perelman's works \cite{P1, P2, P3}, 
all closed orientable $3$-manifolds are classified into four types: 
reducible, toroidal, Seifert fibered, and hyperbolic manifolds. 
Then we can observe that,  generically, 
the structure of a knot complement persists in surgered manifolds. 
Actually the famous Hyperbolic Dehn Surgery Theorem, 
due to Thurston \cite[Theorem 5.8.2]{Th}, 
says that each hyperbolic knot 
(i.e., a knot with hyperbolic complement) 
admits only finitely many Dehn surgeries 
yielding non-hyperbolic manifolds. 
In view of this, such finitely many exceptions 
are called \textit{exceptional surgeries}, 
and giving an interesting subject to study. 

In this note, for a given hyperbolic knot $K$, 
$\mathcal{E}(K)$ denotes 
the set of the slopes along each of which 
the Dehn surgery on $K$ is exceptional, 
and we call $\mathcal{E}(K)$ 
\textit{the exceptional surgery slope set} 
for $K$. 

This set $\mathcal{E}(K)$ is a finite set for each knot $K$, 
but as stated in \cite[Problem 1.77(B)]{K}, 
Gordon conjectured that there exist 
the universal upper bounds on 
the cardinality and the diameter of such sets, 
which are 10 and 8 respectively. 

Here the diameter of $\mathcal{E}(K)$ 
is defined as the maximum of the distance 
(i.e., the minimal intersection number between their representatives) 
between a pair of the elements in $\mathcal{E}(K)$. 

There had been many studies about the conjecture, 
and eventually, in \cite{LM}, 
Lackenby and Meyerhoff gave an affirmative answer 
to the conjecture as follows.

\begin{Theorem*}[{\cite[Theorems 1.1 and 1.2]{LM}}]
Let $K$ be a hyperbolic knot in a closed orientable 3-manifold, 
and $\mathcal{E}(K)$ the exceptional surgery slope set for $K$. 
Then the cardinality of $\mathcal{E}(K)$ is at most 10, 
and the diameter of $\mathcal{E}(K)$ is at most 8.  
\end{Theorem*}

Their proof is computer-aided in part, 
and both bounds are achieved simultaneously. 
We claim in the next section that, 
for $\mathcal{E}(K)$, the diameter bound 8 
actually implies the maximal cardinality bound 10. 
Note that 
such an implication can not hold for general sets of slopes, 
as remarked in \cite[Section 2]{LM}. 
Actually there exists a set of slopes that 
its diameter is 8 but its cardinality is 12. 

\bigskip

The key of our claim is the following known fact.

\begin{Proposition*}
For any hyperbolic knot, 
there exists a slope on the peripheral torus such that 
all exceptional surgery slopes 
have distance at most two from the slope. 
\end{Proposition*}

This fact follows from 
\cite[Theorem 2.5]{Wu} and 
the unpublished result due to Gabai and Mosher 
together with an affirmative answer to Geometrization Conjecture. 
A part of the proof of Gabai-Mosher's theorem 
is included in the unpublished monograph \cite{Mo}. 
See also \cite[Theorem 6.48]{C}. 
An independent proof for the theorem 
is also obtained by Calegari 
as a corollary of \cite[Theorem 8.24]{C}. 

Remark that both proofs by Gabai-Mosher and Calegari 
are very deep results based on the study of the foliation theory. 
In particular, the slope described in the proposition 
comes from an essential lamination in the knot exterior, 
called a \textit{degeneracy slope}, and 
is rather difficult to compute in practice.

Concerning the proposition above, 
we show that, in generic cases, 
the slope represented by the shortest geodesic 
on a horotorus in a hyperbolic knot complement 
can play the same role as that particular slope in Proposition.

\bigskip

We here remark that, 
in \cite{I}, the author showed that, 
if for a hyperbolic knot $K$, 
there exists a slope on the peripheral torus such that 
all exceptional surgery slopes 
have distance at most ONE from the slope, 
then 
the cardinality of $\mathcal{E}(K)$ is at most 10, 
and the diameter of $\mathcal{E}(K)$ is at most 8.

\section{Relationship between upper bounds}

In this section, we give a proof of the following:

\begin{Theorem}
Let $K$ be a hyperbolic knot in a closed orientable 3-manifold. 
Then, for $\mathcal{E}(K)$, 
if the diameter is at most 8, 
then the cardinality is at most 10. 
\end{Theorem}

\begin{proof}
Recall first that it is well-known that 
slopes on a torus are parametrized by 
rational numbers with $1/0$, 
using a meridian-longitude system. 
See \cite{R} for example. 

Now, by virtue of Proposition, for $K$, 
there exists a slope $\gamma$ on the peripheral torus 
such that all exceptional surgery slopes 
have distance at most two from $\gamma$. 
We set this $\gamma$ to be the meridian, 
which corresponds to $1/0$.
It should be noted that this $\gamma$ 
can be an element of $\mathcal{E}(K)$. 

We further set a longitude, which corresponds to $0/1$, 
and then we identify each element in $\mathcal{E}(K)$ 
other than $\gamma$ as an irreducible fraction. 
Recall here that the distance between such a pair of slopes 
$a/b $ and $c/d $ is calculated as $| a d - b c |$. 

Suppose first that there are no integral elements in $\mathcal{E}(K)$, 
equivalently, all the elements in $\mathcal{E}(K)$ have 
distance 2 from $\gamma$. 
Then any pair of elements, say $x/2$ and $y/2$ in $\mathcal{E}(K)$, 
other than $\gamma$ has distance at least 4. 
Together with the assumption that 
the diameter of $\mathcal{E}(K)$ is at most 8, 
we see that the cardinality of $\mathcal{E}(K)$ is at most 4. 

Thus we next suppose that 
$\mathcal{E}(K)$ contains some integral elements. 
In this case, after taking the mirror image if necessary, 
we can set a longitude such that 
integral elements in $\mathcal{E}(K)$ correspond to 
$\{ 0 , \cdots , N_k \}$ 
with $N_i \ge 0$ for $1\le i \le k$. 
Here $k$ denotes the number of integral elements in $\mathcal{E}(K)$. 
We remark that $k \le N_k +1$ holds, and, 
since we are assuming that 
the diameter of $\mathcal{E}(K)$ is at most 8, 
$ N_k \le 8$ holds. 

On the other hand, 
non-integral elements in $\mathcal{E}(K)$ are, if exist, all half integers, 
say $\{ M_1 /2 , \cdots , M_l / 2 \}$ 
with $M_j$ odd for $1\le j \le l$. 
Then we have $ | M_1 - 2 N_k | \le 8$ from the assumption 
that the diameter of $\mathcal{E}(K)$ is at most 8. 
It implies $ - 4 + N_k \le M_1 / 2 \le 4 + N_k $. 
Since $M_1$ is odd, we further obtain that 
$ - 7/2 + N_k \le M_1 / 2 \le 7/2 + N_k $. 
Similarly we have $ - 7/2 \le M_l / 2 \le 7/2 $, 
and then, it follows that 
$M_l / 2 - M_1 / 2 \le 7/2 -(-7/2 + N_k) = 7 - N_k$. 
This implies that
the number of half-integral elements in $\mathcal{E}(K)$ 
is at most $8 -N_k$. 
Since $\mathcal{E}(K)$ consists of integral elements 
and half-integral elements together with $\gamma$, 
the cardinality of $\mathcal{E}(K)$ is 
at most $(N_k + 1 ) + (8 -N_k) + 1 = 10$. 
\end{proof}

As remarked in Section 1, for general sets of slopes, 
such a diameter bound does not imply the required cardinality bound. 
For example, we actually have the following set of slopes: 
$$
\left\{ \ 
\frac{1}{0}  \ ,  \ \frac{0}{1}  \ , \ \frac{1}{1} \  , \ \frac{2}{1} \  , \ 
\frac{3}{1}  \ , \ \frac{3}{2} \  , \ \frac{4}{3} \  , \ \frac{5}{3}  \ , \ 
\frac{5}{4}  \ , \ \frac{7}{4}  \ , \ \frac{7}{5}  \ , \ \frac{8}{5}  \ 
\right\}
$$
By direct calculations, we see that its diameter is 8, but its cardinality is 12.

\section{Existence of particular slope}

In this section, we give 
a proof of the following proposition. 

\begin{Proposition}
If a hyperbolic knot complement contains a horotorus 
of area greater than $8/\sqrt{3}$, then 
the slope on the peripheral torus 
represented by the shortest geodesic on the horotorus 
have distance at most two from 
all exceptional surgery slopes for the knot. 
\end{Proposition}

\begin{proof}
We recall some basic terminologies. 
Let $K$ be a hyperbolic knot in a 3-manifold $M$. 
Then the universal cover of the complement $C_K$ of $K$ 
is identified with the hyperbolic $3$-space $\mathbb{H}^3$. 
Under the covering projection, 
an equivariant set of horospheres 
bounding disjoint horoballs in $\mathbb{H}^3$ 
descends to a torus embedded in $C_K$, 
which we call a \textit{horotorus}. 
As demonstrated in \cite{Th}, 
a Euclidean metric on a horotorus $T$ is obtained by 
restricting the hyperbolic metric of $C_K$. 
By using this metric, the length of a curve on $T$ can be defined. 
Also $T$ is naturally identified with the peripheral torus of $K$, 
since the image of the horoballs under the covering projection 
is topologically $T$ times half open interval. 
Thus, for a slope $r$ on the peripheral torus of $K$, 
we define the \textit{length} of $r$ with respect to $T$ 
as the minimal length of the simple closed curves on $T$ 
which represent the slope on $T$ corresponding to the slope $r$. 

Now suppose that $T$ has the area greater than $8/\sqrt{3}$. 
Let $\gamma$ be the shortest slope on $T$. 

\begin{Claim}
If $\Delta (\gamma , \gamma') \ge 3$ holds for a slope $\gamma'$, 
then the length of $\gamma'$ is greater than $6$.
\end{Claim}

\begin{proof}
Let $h$ be the length of $\gamma$, 
and $w$ 
the length of the shortest path which starts and ends on $\gamma$, 
but which is not homotopic into $\gamma$. 
Then the length of a slope $\gamma'$ 
is at least $w \Delta (\gamma , \gamma')$. 
Thus, to prove the claim, 
it suffices to show that $w > 2$. 

Now we are supposing that 
the area of $T$ is greater than $8/\sqrt{3}$, 
which is equal to $w h$. 
Then, in the case that $h \le 4/\sqrt{3}$, 
we have $w > 8 / \sqrt{3} \cdot 1/h \ge 2$ immediately. 

On the other hand, 
in \cite[Proof of Theorem, page 1049-1050]{M}, 
it is shown that $w \ge h \sqrt{3} /2 $ holds in general. 
Thus, in the case that $h > 4/\sqrt{3}$, 
we have $w \ge h \sqrt{3} /2 > 2$. 

These imply that 
the length of a slope $\gamma'$ is greater than $6$. 
\end{proof}

Finally we use the so-called ``6-Theorem" due to Agol \cite{A} and Lackenby \cite{L} 
together with the affirmative answer to the Geometrization Conjecture, 
given by Perelman \cite{P1,P2,P3}.

\begin{Claim}
Dehn surgery along a slope of 
length greater than $6$ is non-exceptional. \qed
\end{Claim}

This completes the proof. 
\end{proof}


\section*{Acknowledgments}

The author would like to thank David Futer for 
pointing out a gap in the previous version and 
giving him many useful information. 
He also thank Danny Calegari 
for giving him information about the work of Gabai-Mosher.

\bibliographystyle{amsplain}

\end{document}